\documentclass[12pt,reqno]{amsart}
\usepackage{cite,epic,eepic,euscript,verbatim,amsmath,amssymb,amsthm,afterpage,float,bm,stmaryrd,color,ulem}
\usepackage{graphicx,epsfig}
\usepackage{comment}
\newtheorem{theorem}{Theorem}[section]

\newtheorem{remark}{Remark}[section]

\numberwithin{figure}{section}
\numberwithin{table}{section}

\def\XXint#1#2#3{{\setbox0=\hbox{$#1{#2#3}{\int}$}
\vcenter{\hbox{$#2#3$}}\kern-.51\wd0}}

\begin{document}

\title[]{A free energy satisfying finite difference method for Poisson--Nernst--Planck equations
}
\author[H.~Liu and Z.-M. Wang]{Hailiang Liu$^\dagger$ and Zhongming Wang$^\ddagger$\\  \\
 }
\address{$^\dagger$Iowa State University, Mathematics Department, Ames, IA 50011} \email{hliu@iastate.edu}
\address{$^\ddagger$ Florida International University,  Department of Mathematics and Statistics,  Miami, FL 33199}
\email{zwang6@fiu.edu}
\date{August 23, 2013; Revision February 02, 2014.}
\begin{abstract}
In this work we design and analyze a free energy satisfying finite difference method  for solving Poisson-Nernst-Planck equations in a bounded domain.
The algorithm is of second order in space, with numerical solutions satisfying all three desired properties: i) mass conservation, ii) positivity preserving, and iii) free energy satisfying in the sense that these schemes satisfy a discrete free energy dissipation inequality. These ensure that the computed solution is a probability density, and the schemes are energy stable and preserve the equilibrium solutions.  Both one and two-dimensional numerical results are provided to demonstrate the good qualities of the algorithm, as well as effects of relative size of the data given.
\end{abstract}

\maketitle



\section{Introduction}
In this paper,  we are interested in constructing a free energy satisfying numerical method  for solving the initial boundary value problem for the Poisson--Nernst--Planck (PNP)  equations,
 \begin{equation}\label{PNP}
\left\{
  \begin{array}{ll}
  \partial_t c=\nabla \cdot(\nabla c+ c\nabla \psi) & \quad  x\in  \Omega, t>0 \\ 
\Delta \psi=- c&  \quad x\in  \Omega, t>0, \\ 
c(t=0, x) =c^{\rm in}(x), & \; x\in \Omega,\\
(\nabla c+ c\nabla\psi)\cdot \textbf{n}=0,\quad
  \frac{\partial \psi}{\partial  \textbf{n}}=\sigma & \quad \text{ on } \partial\Omega,\; t>0,
 \end{array}
\right.
\end{equation}
where $c$ is the concentration of ion species,  $\Omega \subset \mathbb{R}^d$  denotes a connected closed domain with smooth boundary $\partial\Omega$,  $\psi$ is the potential governed by the Poisson equation which is necessary to determine the electrostatic field,  and $\textbf{n}$ is the unit outward normal vector.  Subject to the given initial and  boundary conditions,  the compatibility condition
 \begin{equation}
 \int_{\partial\Omega} \sigma ds +\int_{\Omega} c^{\rm in} dx =0 \label{compatibility}
 \end{equation}
is necessarily to be imposed for solvability of the problem. By free-energy satisfying we mean that the free energy dissipation law is 
satisfied at the discrete level. 

The PNP equations describe the diffusion of ions under the effect of an electric field that is itself caused by those same ions. The system couples the Nernst-Planck (NP) equation (which describes the drift of ions in a potential gradient by Ohm's law and  diffusion of ions in a concentration gradient by Fick's law)
and the Poisson equation (which relates charge density with electric potential). 
This system of equations  for multiple species has  been extensively used 
in the modeling of semiconductors (see e.g., \cite{MRS90}), and
the membrane transport  in biological ion channels (see e.g., \cite{Ei98}). 

The PNP system can hardly be solved analytically. The main difficulty arises from the nonlinear coupling of the electrostatic potential and concentrations of chemical species.  When the physical domain has a  simple geometry,  a semi-explicit formula was derived in \cite{LLWM} for the steady-state solution; the existence and 
stability of the steady-state solution was established  a long while ago \cite{Jerome}  in the study of the steady Van Roostbroeck model in semiconductors. It has been proved by H. Gajewski and K. G\"{a}rtner in \cite{GG96}  that the solution to the drift-diffusion system  converges to the thermal equilibrium state as time becomes large if the boundary conditions are in thermal equilibrium.  The key-point of the proof is an energy estimate with the control of the free energy dissipation.  Long time behavior was also studied in \cite{BHN},  and further in \cite{AMT, BD} with refined convergence rates.


In the past decade a growing interest in  PNP systems  has been driven mainly by experimental and numerical advances.  Computational algorithms  have been constructed for both simple one-dimensional settings and complex three-dimensional models in various chemical/biological applications, and have been combined with the Brownian Dynamics simulations;  cf.\
\cite{Eisenberg96,Nitzan_BiophysJ99,Kurnikova_BiophysJ00,Eisenberg_Langmuir00,Furini_Etal_BiophysJ06,Chung_Etal_BiophysJ00,LHMZ,LLWM, Roux_PNAS03,Roux_BiophysJ04,Roux_QRevBiophys04,Sejnowski_BiophysJ08,SL,SLL,ZCW}.  Many of these existing algorithms are introduced to handle specific settings in complex applications, in which one may encounter different numerical obstacles, such as discontinuous coefficients, singular charges, geometric singularities, and nonlinear couplings to accommodate various phenomena  exhibited by biological ion channels \cite{WZCX12}.  For instance, \cite{SL,SLL} developed a finite difference scheme for liquid junction and ion selective membrane potentials, in which authors used a fully implicit discretization scheme and the Newton-Raphson solver for the resulting linear system. In \cite{LHMZ}, a 3D finite element method for the system with a singular (point-like) charge outside of the concentration region was developed.  
A second order PNP solver was developed in \cite{ZCW} in a realistic ion-channel context with the Dirichlet boundary condition.  
{In spite of many existing computational studies,  rigorous numerical analysis seems to be still lacking.}

Our objective is to construct and analyze an explicit second-order PNP algorithm to incorporate main mathematical features of the PNP system so that the numerical solution remains faithful for long time simulations, {i.e., the numerical solution possesses desired properties including conservation of ions, positivity of concentration and dissipation of the free energy.}   Therefore we consider only the standard PNP equations of form (\ref{PNP}),   while the reader may find  intensive discussions of the modeling aspect of  the PNP system in the literature.

The main  properties of the solution to (\ref{PNP}) are the nonnegativity principle, the mass conservation and the free energy dissipation, i.e.,
\begin{align}
& c^{\rm in} \geq 0  \Longrightarrow  c \geq 0 \quad \forall t>0, \label{c5positivity} \\
& \int_{\Omega} c(t,x)\,dx=\int_{\Omega} c^{\rm in}(x)\,dx \quad \forall t>0,\\
&  \frac{d}{dt} \tilde F = -\int_{\Omega} c^{-1}|\nabla c +c \nabla \psi|^2 dx +\frac{1}{2}\int_{\Gamma} (\sigma_t \psi -\psi_t \sigma)ds,
\end{align}
where the free energy $\tilde F$ is defined by
$$
\tilde F = \int_{\Omega} \left[ c {\rm log} c+\frac{1}{2}c \psi \right] dx.
$$
If $\sigma$ does not depend on time, we can use a modified functional $ F=\tilde F +\frac{1}{2}\int_{\Gamma} \sigma \psi ds$ so that
$$
\frac{d}{dt} F = -\int_{\Omega} c^{-1}|\nabla c +c \nabla \psi|^2 dx \leq 0.
$$
These properties are also naturally desired for numerical methods solving (\ref{PNP}).  In this paper, we develop such a method. We will demonstrate that these properties of the numerical methods could be critical in obtaining the long-time behavior of the solutions.

It is difficult for numerical schemes to preserve all three properties for PNP equations exactly at the discrete level. A recent effort toward this direction is found in   
\cite{FMELL}, where  the authors present an implicit second order finite difference scheme with a simple iteration so that the energy dissipation law is approximated closely.

The free energy dissipation in time  is also the driving force so that the large time behavior of the solution to (\ref{PNP}) is  governed by  the steady-state solution.
In fact, the steady-state solution is necessarily of the form  $c= Ze^{-\psi}$,  which gives  $\Delta \psi =-Ze^{-\psi}$, and $Z$ can be further
determined by integration of the Poisson equation while using the Neumann boundary data.  In other words, the PNP system near steady states is close to a Poisson--Boltzmann equation (PBE)
$$
\Delta \psi =   \left(\int_{\partial \Omega}\sigma ds\right) \frac{e^{-\psi}}{\int_{\Omega} e^{-\psi}dx}, \quad x\in \Omega, \quad    \frac{\partial \psi}{\partial  \textbf{n}}\Big|_{\partial \Omega} =\sigma.
$$
 Therefore, the numerical method presented in this paper may be used as an iterative algorithm to numerically compute such a nonlocal PBE.  We refer to \cite{LLHLL} for a rigorous study of a nonlocal Poisson--Boltzmann type equation with a small dielectric parameter as derived from the two-species PNP system. In general,  the PBE  describes the electrostatic interaction and ionic density distributions of a solvated system at the equilibrium state. Due to the effectiveness of the PBE for applications in chemistry and biophysics, a large amount of literatures and many solution techniques have been produced in this area and directed to studies of diverse biological processes, see e.g. \cite{Lu_Etal_JCP07, CDLM_JPCB08, DavisMcCammon90, Fixman79, GrochowskiTrylska_Biopoly08,Li_Nonlinearity09,Li_SIMA09,SharpHonig90}.

\subsection{ Related models} 
In a larger context, the concentration equation falls into the general class of  aggregation equations with diffusion
\begin{equation}\label{u}
c_t +\nabla\cdot (c\nabla (G*c))=\Delta c,
\end{equation}
where $G*c=\int G(x-y)c(t, y)dy$.  Such a model has been widely studied in applications such as biological swarms \cite{BCM07, BF08, TBL06} and  chemotaxis  \cite{BRB11, BCL09}.  For chemotaxis,  a wide literature exists in relation to the Keller-Segel model (see \cite{BRB11, BCL09} and references therein).  The left-hand-side in (\ref{u}) represents the active transport of the density $c$ associated to a non-local velocity field $v =\nabla (G*c)$. The potential $G$ is assumed to incorporate attractive interactions among individuals of the group, while repulsive (anti-crowding) interactions are accounted for by the  diffusion in the right-hand-side.

Of central role in studies of model (\ref{u}), and also particularly relevant to the present research, is the gradient flow formulation of the equation with respect to the free energy
\begin{equation}\label{eu}
F[c]=\int_{\Omega} c \log c dx -\frac{1}{2}\int_{\Omega}\int_{\Omega} G(x-y)c(x)c(y)dxdy.
\end{equation}
The free energy contains both entropic part and  the interaction part.  There is a vast literature on entropic schemes for  kinetic equations such as Fokker-Planck type equations. For these equations,  information carried by the probability density  becomes less and less as time evolves, the probability density is expected to converge to the equilibrium solution in a closed system regardless of how initial data are distributed. The entropy dissipation in time is the underlying mechanism for this phenomenon. The entropy satisfying methods recently developed  in \cite{LY12, LY13}  are the main motivation for this paper.

\subsection{Contents}  This paper is organized as follows: in Section 2, we describe the algorithm for the one dimensional case. Theoretical analysis for both semi-discrete and full discrete schemes is provided. We also discuss how the Poisson equation is solved.  Numerical results of both one and two dimensions are presented in Section 3. Finally, in Section 4, concluding remarks are given.

\section{The numerical method}
\subsection{Reformulation}
Following  \cite{LY12},  we formally reformulate the system by setting
$$g(t,x)=c(t,x)e^{\psi(t,x)}$$
to obtain
\begin{align}\label{cg}
\partial_t c& =\nabla \cdot(e^{-\psi}\nabla g)\quad  x\in  \Omega, \; t>0,  \\ \label{gn}
 \frac{\partial g}{\bf \partial n} &=0.
\end{align}
With this reformulation,  an iterative algorithm may be designed as follows:  given $c$ we obtain $\psi$ by solving the Poisson equation
$$
-\Delta \psi=c, \quad    \frac{\partial \psi}{\partial  \textbf{n}}=\sigma,
$$
which determines $g$; we further update  $c$ by solving (\ref{cg}), (\ref{gn}).

We  shall describe our numerical algorithm in one dimensional setting only, with one or multi-species.
The algorithm is extensible to high dimensional case in a straightforward dimension by dimension manner.

In the one dimensional case with $\Omega=[a, b]$, 
the above two steps correspond to solving the following two sets of problems:

{ 
\begin{align}
\psi_{xx}& =-c  \quad  x\in [a, b],\\
\psi_x(t, a)& =-\sigma_a, \quad \psi_x(t, b)=\sigma_b,\\
c_t & =(e^{-\psi}g_x)_x \quad  x \in [a, b], \; t>0,\\
 c(t=0, x) & =c^{\rm in}(x), \quad x\in [a, b],\; t>0,\\
g_x(t, a) & =g_x(t,b)=0.
\end{align}
Note that we set $\psi_x(t, a)$ and $\psi_x(t, b)$ to different constants.
Integrating the Poisson equation and using the mass conservation for $c$ we obtain
$$
\sigma_b+\sigma_a= -\int_a^b c^{\rm in}(x)dx,
$$
which is exactly the compatibility condition stated in (\ref{compatibility}).

If $\sigma_a, \sigma_b$ are independent of time, we  set
$$
 F=\tilde F + \frac{1}{2}( \sigma_b \psi(t, b) + \sigma_a \psi(t, a)).
$$
A direct calculation gives
$$
\frac{d}{dt}  F = - \int_a^b  c^{-1}  (c_x+c\psi_x)^2dx \leq 0.
$$
We now describe our algorithm by first  partitioning the domain $[a, b]$ with   $x_{1/2}=a$,   $x_{N+1/2}=b$, and interior grid points
$$
x_j=a+h(j-1/2), \quad  j=1,\cdots, N.
$$

\subsection{Algorithm}

\begin{itemize}
\item[1.] We use $c_j$ to approximate $c(t, x_j)$ and $\psi_j$ to approximate $\psi(t, x_j)$, respectively.  Given $c_j, j=1, \cdots,  N$, we compute the potential $\psi$ by
\begin{equation}\label{pc}
\frac{\psi_{j+1}-2\psi_j+\psi_{j-1}}{h^2}=-c_j,
\end{equation}
where $\psi_1-\psi_0=-\sigma_a h$, and $\psi_{N+1}-\psi_{N}=\sigma_b h$.
For definiteness,  we  set  $\psi_1=0$ at any time $t$ to single out  a particular solution since $\psi$ is unique up to an additive constant.
\item[2.] With the above obtained $\psi_j, j=1,\cdots, N$, the semi-discrete approximation of  the concentration $c$ satisfies
\begin{align}\label{tc}
\frac{d}{dt}c_j&=\frac{1}{h} \left(e^{-\psi_{j+\frac{1}{2}}}\widehat{g}_{x,j+\frac{1}{2}} -e^{-\psi_{j-\frac{1}{2}}}\widehat{g}_{x,j-\frac{1}{2}}\right):=Q_j(c,\psi),\\ \nonumber
\psi_{j+\frac{1}{2}}&=\frac{\psi_{j+1}+\psi_j}{2},\\ \nonumber
\widehat{g}_{x,j+\frac{1}{2}}&=\frac{g_{j+1}-g_j}{h}=\frac{c_{j+1}e^{\psi_{j+1}}-c_je^{\psi_j}}{h},
\end{align}
 and {$\widehat{g}_{x, 1/2}=0$, and ${\widehat{g}_{x, N+1/2}=0}$.}
\item[3.] Discretize $t$ uniformly: $t_n=t_0+kn$, so $c_j^n \sim c(t_n,x_j)$ and $\psi_j^n \sim \psi(t_n,x_j)$ satisfy
\begin{equation}\label{cq}
\frac{c_j^{n+1}-c_j^n}{k}=Q_j(c^n,\psi^n).
\end{equation}
\end{itemize}
}
For the system case with the vector unknown ${\bf c}=[c^1, c^2,\cdots, c^m]^{\rm T}$, the PNP system in a simple form is 
\begin{subequations}\label{mc}
\begin{align}
  \partial_t c^i & =\nabla_x \cdot \left( \nabla_x  c^i + q_ic^i \nabla_x  \psi  \right), \quad i=1,2,\cdots, m, \\
\Delta \psi & =- \sum_{i=1}^mq_ic^i,
\end{align}
\end{subequations}
where $q_i$ is the charge of $c^i$.  The one-dimensional algorithm needs to be modified:  (\ref{pc}) is replaced by
\begin{equation}\label{pc+}
\frac{\psi_{j+1}-2\psi_j+\psi_{j-1}}{h^2}=-\sum_{i=1}^m q_ic^i_j,
\end{equation}
where $c_{j}^i\sim c^i(t,x_j)$,  in order to solve the Poisson equation (\ref{mc}b), and each $c^i_j$ is obtained from solving
\begin{align}\label{tc+}
\frac{d}{dt}c^i_j=Q_j(c^i, q_i\psi),
\end{align}
where in the formula for $Q_j$ the potential is replaced by $q_i\psi$.  For the multi-dimensional case,  the algorithm can be applied in a dimension by dimension manner. We illustrate these cases by numerical examples later.

\begin{remark} The above algorithm can be easily modified to solve the PNP equations with more physical parameters: 
\begin{subequations}\label{mc+}
\begin{align}
  \partial_t c^i & =\nabla_x \cdot \left(D_i  \nabla_x  c^i + \frac{q_ic^i}{k_B T}  \nabla_x  \psi  \right), \quad i=1,2,\cdots, m, \\
\nabla_x\cdot(\epsilon \nabla_x  \psi) & =- (\rho_0+ \sum_{i=1}^mq_ic^i),
\end{align}
\end{subequations}
where $c^i$ is the ion density for the $i$-th species, $q_i$ is the charge of $c^i$, $D_i$ is the diffusion constant, $k_B$ is the Boltzmann constant, $T$ is the absolute temperature, $\epsilon$  is the permittivity, $\psi$ is the electrostatic potential, $\rho_0$ is 
the permanent (fixed) charge density of the system, and $m$ is the number of ion species. 
\end{remark}

\begin{remark} It is also possible to generalize the schemes presented in this work to second order finite difference schemes on non-uniform meshes, yet the analysis would appear more complicated. 
\end{remark}

\subsection{Properties of the numerical method}
The numerical solution obtained from  the above  algorithm has some  desired  properties as stated in the following.
\begin{theorem}
\begin{itemize}
\item[1.]  Both semi-discrete scheme (\ref{tc}) and Euler forward discretization (\ref{cq})  are conservative in the sense that  the total concentration $c_j$ remains unchanged in time,
\begin{align}
\label{conserve1}
\frac{d}{dt} \sum_jc_j h & =0, \quad t>0\\ \label{conserve2}
\sum_jc_j^{n+1}h & =\sum_jc_j^n h.
\end{align}
\item[2.]  The discrete concentration $c_j^n$ remains  positive in time:  if $c_j^n \geq 0,  j=1,\cdots, N$,  then
$$
c_j^{n+1}\geq 0$$
provided the condition $k< h \lambda_0 $ where
\begin{equation}\label{la}
\lambda_0= \frac{1}{ e^{\frac{-h\sigma_b}{2}} + e^{\frac{-h\sigma_a}{2}}}.
\end{equation}
\item[3.] The semi-discrete {free energy}
 $$
 F=\sum_j \left(c_j\ln c_j+\frac{1}{2}c_j\psi_j\right)h +  \frac{1}{2}\sigma_a\psi_1+ \frac{1}{2} \sigma_b\psi_N
 $$
 satisfies
 \begin{equation}\label{fd}
\frac{d}{dt} F=-\frac{1}{h}\sum_{j=1}^{N-1}e^{-(\psi_{j+1}+\psi_j)/2}\left(\ln g_{j+1}-\ln g_j\right)(g_{j+1}-g_j)\leq 0,
\end{equation}
therefore nonincreasing.
\end{itemize}
\end{theorem}
\begin{proof}
\begin{itemize}
\item[1.]  We first prove the conservation \eqref{conserve1} and  \eqref{conserve2} . Since
   \begin{align*}
   \frac{d}{dt}c_j&=\frac{1}{h} \left(e^{-\psi_{j+\frac{1}{2}}}\widehat{g}_{x,j+\frac{1}{2}} -e^{-\psi_{j-\frac{1}{2}}}\widehat{g}_{x,j-\frac{1}{2}}\right),
\end{align*}
we have
\begin{align}
   \frac{d}{dt} \sum_{j=1}^Nc_j h &=e^{-\psi_{N+\frac{1}{2}}}\hat{g}_{x,N+\frac{1}{2}} -e^{-\psi_{\frac{1}{2}}}\hat{g}_{x,\frac{1}{2}}=0,
    \label{conservepf}
\end{align}
where we have used the zero flux boundary condition  $\hat{g}_{x,\frac{1}{2}}=0$ and $\hat{g}_{x,N+\frac{1}{2}}=0$. Similarly, summation of (\ref{cq}) over $j$ gives
$$
C=\sum_jc_j^{n+1}h =\sum_jc_j^n h.
$$
\item[2.]  Define $A^n_j=\psi_{j+1}^n-\psi_j^n$,  the boundary condition gives $A_0^n=-h\sigma_a$ and $A_N^n=h\sigma_b$. Let mesh ratio be denoted by $\lambda=k/h$,  we can rewrite (\ref{cq}) as
$$
c_j^{n+1}
=c^n_j\left( 1- \lambda \left( e^{\frac{-A_j^n}{2}} + e^{\frac{A_{j-1}^n}{2}} \right) \right)+\lambda c_{j+1}^n e^{\frac{A_j^n}{2}}  + \lambda  c_{j-1}^n e^{\frac{- A_{j-1}^n}{2}} .
$$
From the discrete Poisson equation
   $$A_j^n-A_{j-1}^n=-h^2c^n_j,$$
   it follows that
\begin{equation}\label{Aj}
A_j^n=\sum_{s=1}^j -h^2c^n_s+A^n_0=-\sum_{s=1}^j h^2c^n_{s} - h \sigma_a
\end{equation}
and $A_N^n=A_0^n-hC$.
Hence $A_j^n$ is decreasing  in $j$ and
$$
h\sigma_b \leq A_j^n \leq -h\sigma_a.
$$
Thus,  we have $ c_j^{n+1}\geq 0$ if  $\lambda \leq \lambda_0$ as  defined in (\ref{la}).
\item[3.]  A direct calculation using $\sum_{j=1}^N \dot{c}_j =0$ gives
\begin{align*}
\dot{F} &= h \sum_{j=1}^N (1+\ln c_j +\psi_j )\dot{c}_j  +\frac{h}{2} \sum_{j=1}^N \left( {c}_j\dot{\psi}_j-\dot{c}_j\psi_j\right)+ {\frac{1}{2}\sigma_a\dot{\psi_1}+\frac{1}{2}\sigma_b\dot{\psi_N}}\\
    &  = h \sum_{j=1}^N (\ln c_j+\psi_j)\dot{c}_j  +\frac{h}{2}\sum_{j=1}^N \left( {c}_j\dot{\psi}_j-\dot{c}_j\psi_j\right)+ {\frac{1}{2}\sigma_a\dot{\psi_1}+\frac{1}{2}\sigma_b\dot{\psi_N}}.
\end{align*}
We proceed with
\begin{align}
   h  \sum_{j=1}^N (\ln c_j+\psi_j)\dot{c}_j &= \sum_{j=1}^N\ln g_j  \left(e^{-\psi_{j+\frac{1}{2}}}\widehat{g}_{x,j+\frac{1}{2}} -e^{-\psi_{j-\frac{1}{2}}}\widehat{g}_{x,j-\frac{1}{2}}\right)    \notag \\
    &  =\frac{1}{h}\sum_{j=1}^N \ln g_j\left( e^{-(\psi_j+\psi_{j+1})/2}(g_{j+1}-g_j)-e^{-(\psi_j+\psi_{j-1})/2}(g_{j}-g_{j-1})\right) \notag \\
    &=-\frac{1}{h}\sum_{j=1}^{N-1}e^{-(\psi_{j+1}+\psi_j)/2}\left(\ln g_{j+1}-\ln g_j\right)(g_{j+1}-g_j)\notag \\
    &\leq 0,  \label{entropy}
\end{align}
since $\widehat{g}_{x, \frac{1}{2}}=\widehat{g}_{x,N+\frac{1}{2}}=0$
and $(\ln\alpha-\ln\beta)(\alpha-\beta)\geq 0$ for any $\alpha>0$ and $\beta>0$.
Next we use the boundary conditions  $\psi_{N+1}=\psi_N+h\sigma_b$ and $\psi_0=\psi_1+h{\sigma_a}$ to obtain
\begin{align}
    \frac{h}{2}\sum_{j=1}^N \left( {c}_j\dot{\psi}_j-\dot{c}_j\psi_j\right)
    & = \frac{-1}{2h}\sum_{j=1}^N \left( (\psi_{j+1}-2\psi_j+\psi_{j-1})\dot{\psi}_j-{(\dot{\psi}_{j+1}-2\dot{\psi}_j+\dot{\psi}_{j-1})}\psi_j\right)\notag\\
 &\textcolor{black}{=\frac{1}{2h}(\psi_1\dot{\psi}_0-\dot{\psi}_1\psi_0+\dot{\psi}_{N+1}\psi_N-\psi_{N+1}\dot{\psi}_N) } \notag\\
 &= {-\frac{1}{2}\sigma_a\dot{\psi_1}-\frac{1}{2}\sigma_b\dot{\psi_N}}. \label{potential}
\end{align}
Putting all together leads to (\ref{fd}).

\vspace{0.1in}


\end{itemize}
\end{proof}
\begin{remark}
For multi-species case governed by \eqref{mc}, the algorithm still preserves three desired properties.  Mass conservation follows from \eqref{tc+} for each $c^i_j$ so that
$$
\sum_{j=1}^N c_j^{i, n}h = \sum_{j=1}^N c_j^{i, 0}h.
$$
For the positivity property, we need to bound  $A_j^n$, which in multi-species case becomes
\begin{equation*}\label{Ajm}
A_j^n = -\sum_{s=1}^j h^2\sum_{i=1}^m q_i c^{i, n}_j - h \sigma_a \leq -h(C^- +\sigma_a),
\end{equation*}
where $C^\pm=h \sum_{\{i,  \pm q_i > 0\}}\sum_{j=1}^N  q_i c^{i, 0}_{j}$. Using the fact that $  \sigma_a+\sigma_b= -(C^++C^-)$ we have
$$
A_j^n \geq -h(C^+ + \sigma_a)= h(C^-+\sigma_b). 
$$
These bounds of  $A_j^n$ ensure that $c^{i, n}_{j}>0$  provided that
$$
\lambda <e^{hC^-/2} \lambda_0,
$$
where $\lambda_0$ is given in (\ref{la}). This sufficient condition is consistent with the one-specie case $C^-=0$.  For the free energy dissipation property,
we consider the semi-discrete functional of the form
$$
 F(t) =\sum_{i=1}^m \sum_{j=1}^N  \left( c_j^i \ln c_j^i +\frac{1}{2}q_ic_j^i \psi_j\right)h +  \frac{1}{2}\sigma_a\psi_1+ \frac{1}{2} \sigma_b\psi_N.
 $$
Following the calculation in \eqref{entropy} for each $c^i_j$, and \eqref{potential} using (\ref{pc+}) for the potential part, we arrive at
 \begin{equation}\label{fd+}
\frac{d}{dt} F=-\frac{1}{h}\sum_{i=1}^m \sum_{j=1}^{N-1}e^{-q_i(\psi_{j+1}+\psi_j)/2}\left(\ln g^i_{j+1}-\ln g^i_j\right)(g^i_{j+1}-g^i_j)\leq 0,
\end{equation}
where $ g_j^i=c_j^ie^{q_i\psi_j}$.
\end{remark}

\subsection{On solving the Poisson equation} \label{onPE}
One of the main numerical tasks is to solve the Poisson equation
$$ \Delta \psi= -c$$
in the first step of our algorithm, subject to the Neumann conditions. In one dimensional case,  we use the numerical boundary condition $\psi_1-\psi_0=-\sigma_a h$ and $\psi_{N+1}-\psi_{N}=\sigma_b h$ to construct the tridiagonal matrix $A$ and the right hand side source term $u$
\begin{center}
\begin{tabular}{cc}
$A=\left[\begin{tabular}{rrrrrrr}
-1 & 1 &  &  &  & & \\
1 & -2 & 1 &  &  & & \\
 & 1 & -2 & 1 &  &  &\\
 & &  & $\cdots$ &  &  &  \\
 &  &  &  1&  -2&1  \\
 &  &  &  &  1& -2 & 1\\
 &  &  &  &  & 1 & -1
\end{tabular}\right], $&
$u=\left[\begin{tabular}{c}
$-c_1h^2-\sigma_a h$ \\
$-c_2h^2$ \\
 $\cdot$\\
 $\cdot$  \\
   $\cdot$\\
 $-c_{N-1}h^2$\\
 $-c_Nh^2-\sigma_b h$
\end{tabular}\right],$
\end{tabular}\\
\end{center}
so that the corresponding linear system becomes $A\psi=u$ with $\psi=[\psi_1, \cdots, \psi_N]^\top$.
 However, the above defined $A$ is singular since the solution to the Poisson equation with Neumann boundary conditions is  unique up to an additive constant. We pick one definite solution by setting $\psi_1=0$, i.e., we set
\begin{center}
\begin{tabular}{cc}
$A=\left[\begin{tabular}{rrrrrrr}
1 &  &  &  &  & & \\
1 & -2 & 1 &  &  & & \\
 & 1 & -2 & 1 &  &  &\\
 & &  & $\cdots$ &  &  &  \\
 &  &  &  1&  -2&1  \\
 &  &  &  &  1& -2 & 1\\
 &  &  &  &  & 1 & -1
\end{tabular}\right]$&
$u=\left[\begin{tabular}{c}
$0$ \\
$-c_2h^2$ \\
 $\cdot$\\
 $\cdot$  \\
   $\cdot$\\
 $-c_{N-1}h^2$\\
 $-c_Nh^2-\sigma_b h$
\end{tabular}\right].$
\end{tabular}\\
\end{center}
In this new formulation, the surface charge $\sigma_a$ is implicitly used because of the compatible condition
 $$
 - \sum_{i=1}^Nc_i h =\sigma_a+\sigma_b,
 $$
which can be seen from  the sum of (\ref{pc}) over $j=1,\cdots, N$. In the old formulation, the first equation gives
$\psi_2=-c_1h^2 -\sigma_a h$ after we set $\psi_1=0$. In the new formulation, the same $\psi_2$ can be obtained
by summing all equations except the first one:
$$
-\psi_2 =-\sum_{i=2}^Nc_ih^2 -\sigma_b h =c_1h^2 +\sigma_a h.
$$

In two dimensions of a rectangle domain $[a,b]\times[a',b']$ with a uniform partition of $x_i=a+h(i-1/2),\quad y_j=a'+h(j-1/2),\quad i=1,\cdots,N_x$ and $j=1,\cdots,N_y$,  we first arrange $\psi_{i,j}$ and $c_{i,j}$ into an array row by row, i.e.,
$$\tilde{\psi}=[\psi_{1,1},\cdots,\psi_{1,N_y},\cdots,\psi_{N_x,1},\cdots,\psi_{N_x,N_y}]^{\rm T},$$
$$\tilde{c}=[c_{1,1},\cdots,c_{1,N_y},\cdots,c_{N_x,1},\cdots,c_{N_x,N_y}]^{\rm T},$$
where $\psi_{i,j}$ and $c_{i,j}$ approximate  $\psi(x_i,y_j)$ and $c(x_i,y_j)$,  respectively.

With the Neumann boundary condition $\frac{\partial \psi}{\partial \bf n}=\sigma$, $\tilde{\psi}_{1}=0$ and in a square domain($N_x=N_y=N$), we solve the linear system $A\psi=u$, where
{ \begin{center}
\begin{tabular}{ll}
${\rm A}=\left[\begin{tabular}{lllll}
${\rm B}_1$ & I &  &  &   \\
I & ${\rm B}_2$ & I &  &   \\
 & &   $\cdots$ &&    \\
  &  &  I& ${\rm B}_2$ & I \\
  &  &  & I & ${\rm B}_3$
\end{tabular}\right]$ &

${\rm B}_1=\left[\begin{tabular}{rrrrr}
1&   &  &  &   \\
  1& -3 & 1 &  &   \\
&  & $\cdots$ &  &   \\
 &  &  1& -3 & 1\\
  &  &  & 1 & -2
\end{tabular}\right]$ \\
\noalign{\vskip 2mm}
${\rm B}_2=\left[\begin{tabular}{rrrrr}
-3& 1  &  &  &   \\
  1& -4 & 1 &  &   \\
&  & $\cdots$ &  &   \\
  &  &  1& -4 & 1\\
  &  &  & 1 & -3
\end{tabular}\right]$&
${\rm B}_3=\left[\begin{tabular}{rrrrr}
-2& 1  &  &  &   \\
  1& -3 & 1 &  &   \\
&  & $\cdots$ &  &   \\
  &  &  1& -3 & 1\\
  &  &  & 1 & -2
\end{tabular}\right],$

\end{tabular}\\
\end{center}
}
and
\begin{align*}u=[&0,-c_{1,2}h^2-\sigma h,\cdots,-c_{1,N-1}h^2-\sigma h,-c_{1,N}h^2-2\sigma h,\\
                             &-c_{2,1}h^2-\sigma h,-c_{2,2}h^2,\cdots,-c_{2,N-1}h^2,-c_{2,N}h^2-\sigma h,\\
                             &\qquad\qquad\qquad\qquad\qquad\cdots\\
                              &-c_{N-1,1}h^2-\sigma h,-c_{N-1,2}h^2,\cdots,-c_{N-1,N-1}h^2,-c_{N-1,N}h^2-\sigma h,\\
&c_{N,1}h^2-2\sigma h, -c_{N,2}h^2-\sigma h, \cdots,c_{N,N-1}h^2-\sigma h,c_{N,N}h^2-2\sigma h]^{\rm T}.
\end{align*}
Note that if $N_x\neq N_y$, the matrices $A$, $B_1$, $B_2$, $B_3$ and $I$ are no longer squares,  and need to be modified accordingly. Note that the matrix $A$ is a constant matrix in our setting.

\section{Numerical examples}
We now test our method in various settings.
\subsection{One dimensional numerical tests}
\subsubsection{Single species}
We first consider a neutral system with $\sigma_a=-1$, $\sigma_b=0$, and $\int_a^b c^{\rm in}(x)dx=1$, which satisfies the compatibility condition \eqref{compatibility}.  The computational domain is $[0,1]$ and initial conditions are
\begin{align*}&(1) \quad c^{\rm in}(x)=1,\\
      &(2) \quad c^{\rm in}(x) =2-2x,\\
      & (3) \quad  c^{\rm in}(x) =  \begin{cases} 2, & \mbox{} 0.5\leq x\leq1 \\ 0, & \mbox{else. } \end{cases}
\end{align*}

Table \ref{1d1errortable} shows the numerical convergence at time $t=0.5$.  {The numerical solutions with $h=0.003125$ are taken to be  discrete reference solutions, and a piecewise cubic spline interpolation is used to obtain the reference solution $c_{ref}$ and $\psi_{ref}$. The $l^{\infty}$ error of the numerical solution $c_h$ with step size $h$ is then defined by $e_c(h)=\max_{x_i}|c_h-c_{ref}|$, where $x_i$ are the grid points. The convergence orders are defined by $\log_2 \frac{e_c(h)}{e_c(h/2)}$. The same definition applies to $\psi$.} We observe that the spacial accuracy of the numerical scheme is  of second order in both $c$ and $\psi$ for all three initial data,  including the discontinuous data.

Fig. \ref{1d1} shows the initials and final steady-state solutions, where we notice that solutions with all initials converge to the same steady-state solution given the same stop criteria, if the boundary data  $\sigma$ and the total density $\int_a^b c^{\rm in}(x)dx$ are the same. The initial $c^{\rm in}$ in (3) is set to be zero in $x<0.5$, the numerical solution will maintain positivity all the time and handle the discontinuity at $x=0.5$ well.   Fig. \ref{1d1F} shows the free energy decay in time, where the free energy $F$ is truncated in $[0.15, 0.2]$ for better comparisons. By adopting the same convergence criteria in the energy $F$, we observe that it takes the longest time for the initial data in (3)  to evolve into  the steady state. This is  as anticipated,   since  the initial in (3) is farther away  from the steady state than other initial data.  Eventually the free energy $F$ converges to $0.15375$ for all three cases.  Long time solution behavior is also tested,   {and our results show that the energy stays at  $0.15375$ even at $t=10$ for all  three cases. }

\begin{figure}[htb]
\caption{Different initials and their associated steady-state solutions}  \label{1d1}
 \includegraphics[height=0.301\textheight,width=\textwidth]{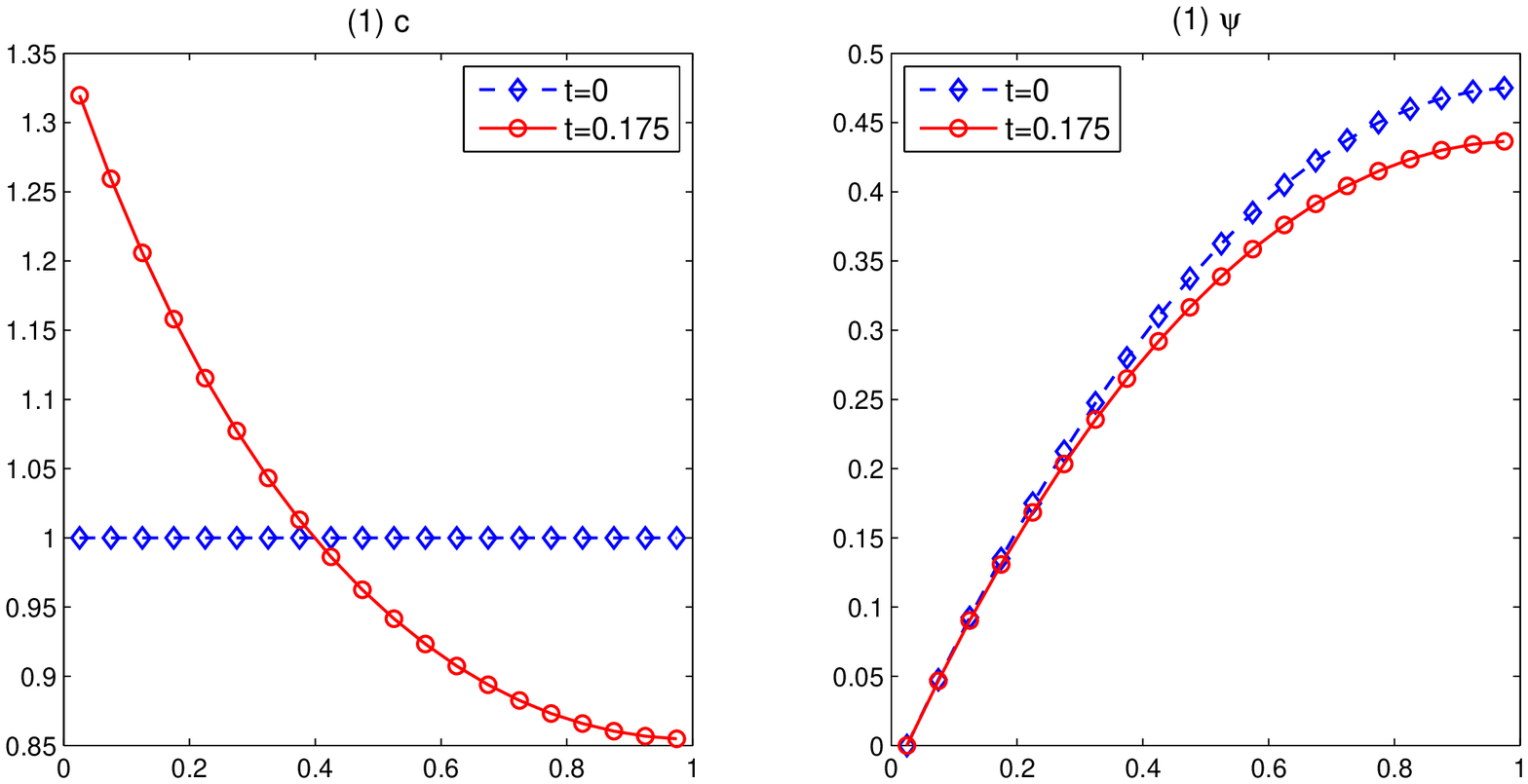}\\
 \includegraphics[height=0.301\textheight,width=\textwidth]{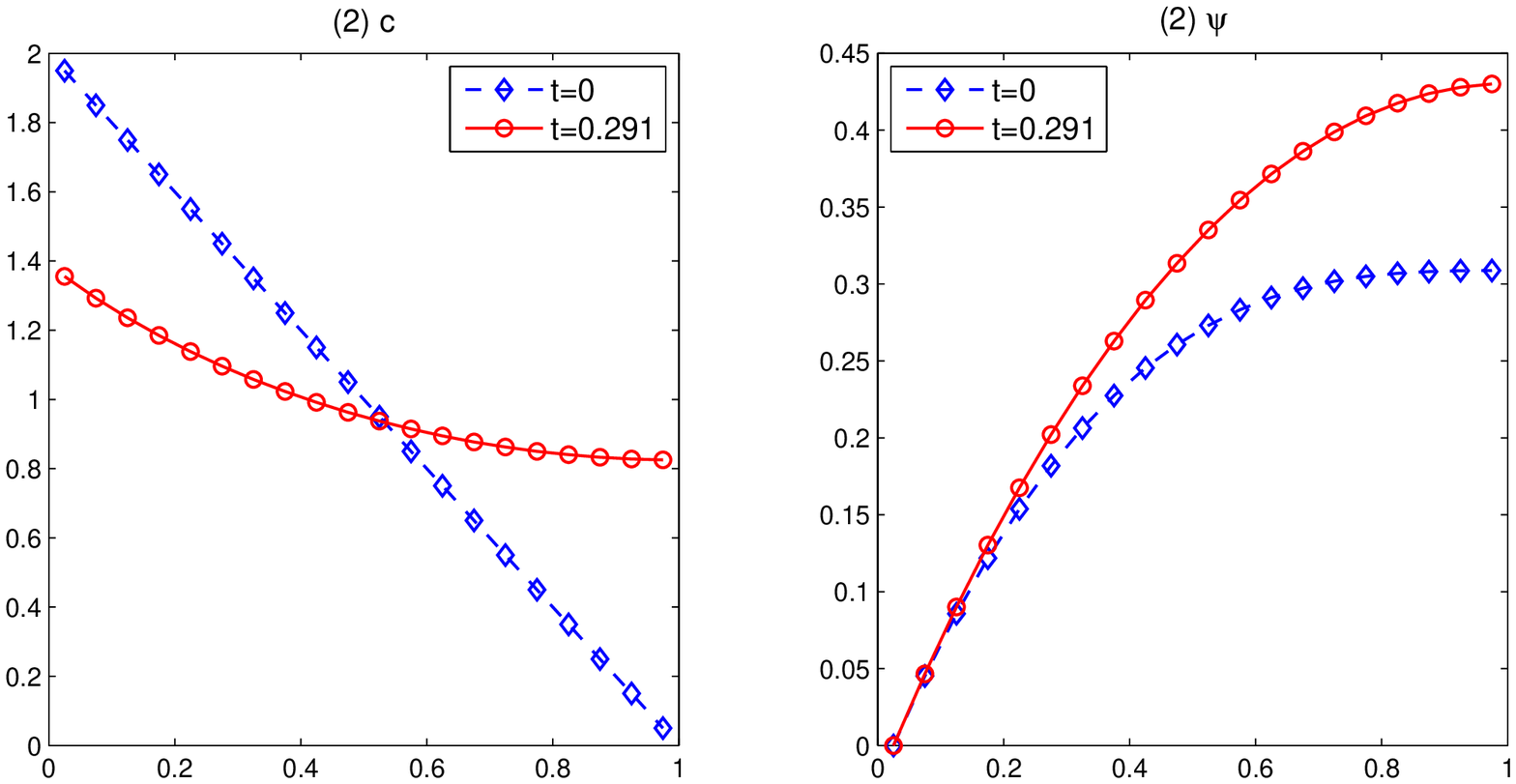}\\
 \includegraphics[height=0.301\textheight,width=\textwidth]{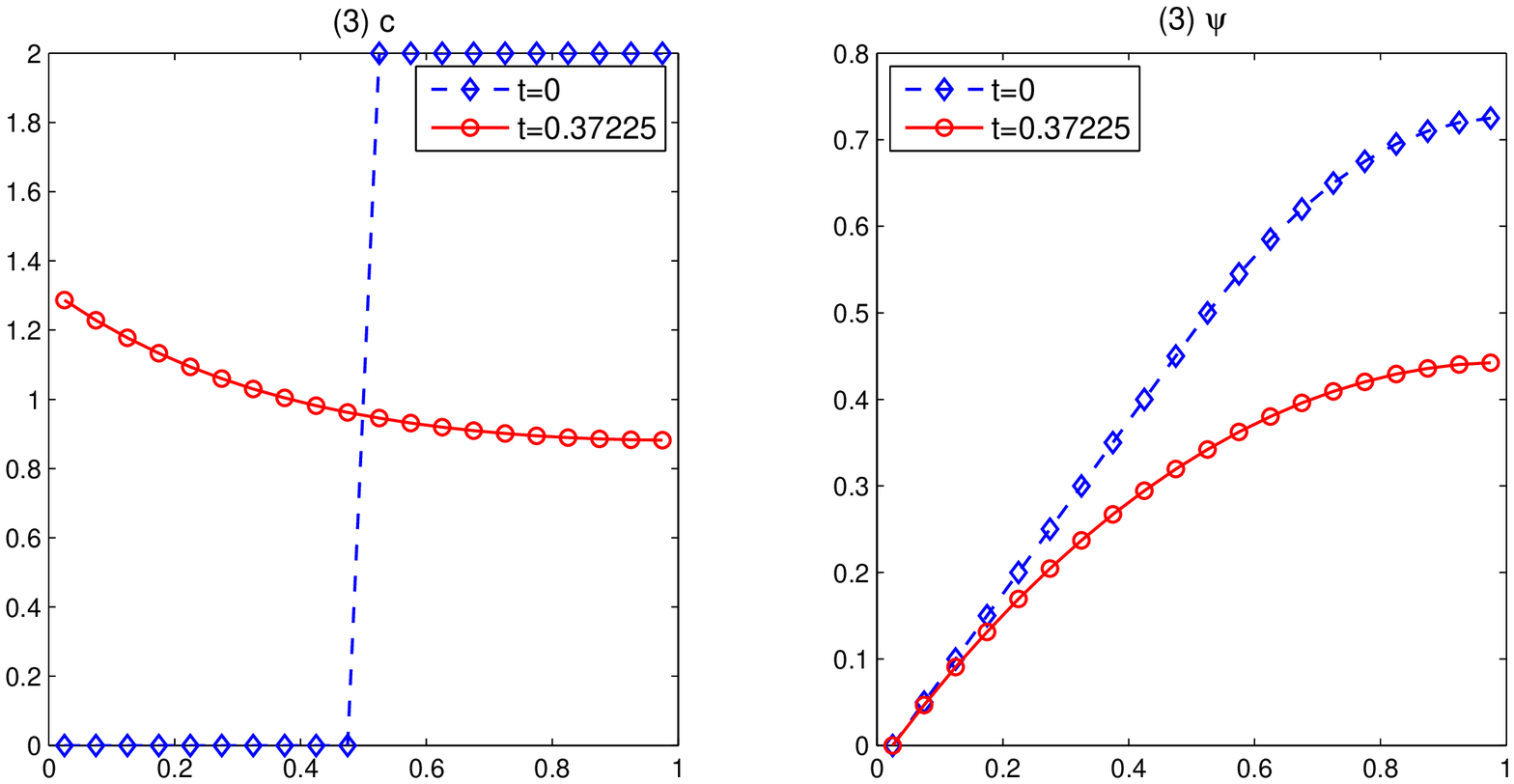}
\end{figure}

\begin{figure}[htb]
\caption{ Free energy decay in time}  \label{1d1F}
\includegraphics[height=0.4\textheight,width=\textwidth]{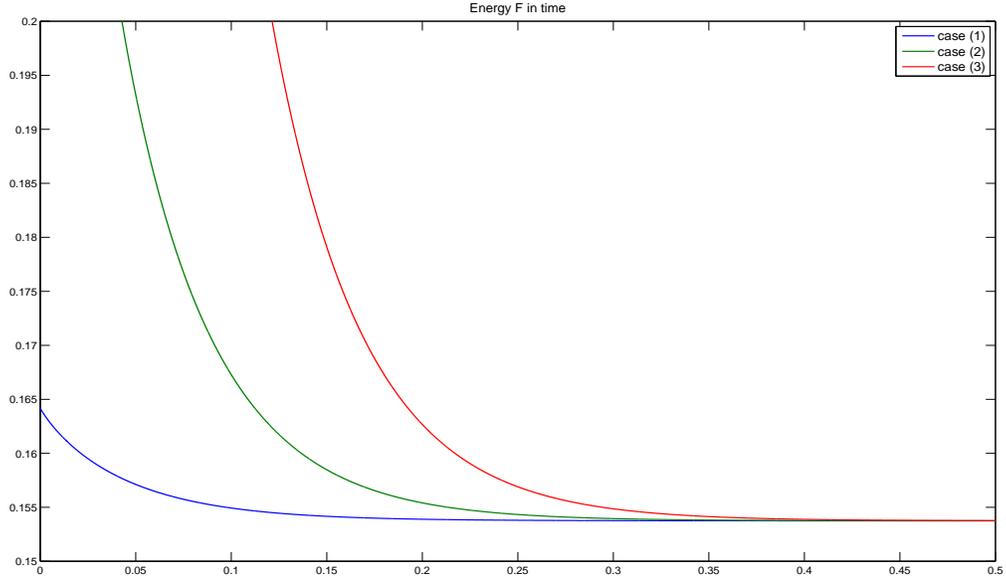}
\end{figure}

\begin{table}[htb]
\caption{$l^\infty$ error table of $c$ and $\psi$ for all initial cases at time $t=0.5$  }    \label{1d1errortable}
\begin{tabular}{|c|c|c|c|c|}
  \hline
h&       error in $c$ & order &error in $\psi$    &order           \\   \hline
          0.2&    0.0026321  &        - & 0.00041461   &        -   \\   \hline
          0.1&  0.00065636   &    2.0037&  0.00012431  &    1.7378  \\   \hline
         0.05&   0.0001639   &    2.0017& 3.4138e-005  &    1.8645   \\   \hline
        0.025& 4.2017e-005   &    1.9637& 9.1157e-006  &    1.9049   \\   \hline
       0.0125& 1.0026e-005   &    2.0672&  2.224e-006  &    2.0352   \\   \hline\hline

          0.2&    0.002514 &          -&  0.00030611&           -   \\   \hline
          0.1&  0.00060573 &     2.0532&   0.0001046&      1.5492    \\   \hline
         0.05&  0.00015138 &     2.0005& 2.9183e-005&      1.8416    \\   \hline
        0.025& 4.1626e-005 &     1.8627& 6.7804e-006&      2.1057     \\   \hline
       0.0125& 9.9647e-006 &     2.0626& 1.6558e-006&      2.0338    \\   \hline
      0.00625& 2.0458e-006 &     2.2842&  3.277e-007&      2.3371     \\   \hline \hline

          0.1&  0.00082461&           - & 0.00014438&          -    \\   \hline
         0.05&  0.00020753&      1.9904 &3.9146e-005&     1.8829    \\   \hline
        0.025&  6.347e-005&      1.7092 &1.2189e-005&     1.6833     \\   \hline
       0.0125& 1.5296e-005&       2.053 & 2.978e-006&     2.0331     \\   \hline
      0.00625& 3.2147e-006&      2.2504 &6.2711e-007&     2.2476     \\   \hline

\end{tabular}
\end{table}
\subsubsection{Multiple species}
We consider a neutral system with two species on $[a,b]$, e.g.,
\begin{align*}
& \partial_t c_i=\nabla \cdot(\nabla c_i+ q_ic_i\nabla \psi)\quad  \text{ in } \quad (a, b), \; i=1,2, \\
&\Delta \psi=- q_1c_1-q_2c_2  \quad \text{   in } (a,b) \\
&(\nabla c_i+q_ic_i\nabla\psi)\cdot {\bf n}|_{\Gamma}=0,\quad
\left.\frac{\partial \psi}{\partial {\bf n}}\right |_{a}=\sigma_a,\quad \left.\frac{\partial \psi}{\partial {\bf n}}\right |_{b}=\sigma_b.\notag
\end{align*}
In a domain [0,1], we pick $q_1=1$ and $q_2=-1$ and $\int_0^1 c_1(0,x)dx=2$, $\int_0^1 c_2(0,x)dx=1$, $\sigma_a=-1$ and $\sigma_b=0$ in order to satisfy compatibility condition (\ref{compatibility}). Initial conditions are
\begin{align*}&(1) \quad c_1(0,x)=2, \quad c_2(0,x)=1, \\
      &(2) \quad c_1(0,x)=4-4x,\quad c_2(0,x)=2x,\\
      & (3) \quad  c_1(0,x)=4x, \quad c_2(0,x)=2-2x.
\end{align*}
 Fig. \ref{1d2} shows the initial conditions and their steady-state solutions.  We also observe that in all cases the numerical solutions converge to the same steady-state solution  for the above fixed $\sigma$, $\int_0^1 c_1(0,x)dx$ and $\int_0^1 c_2(0,x)dx$, and the free free energy  $F$ converges to $1.5147$.
\begin{figure}[h]
\caption{Different initials and their associated steady-state solutions} \label{1d2}
 \includegraphics[width=1.05\textwidth]{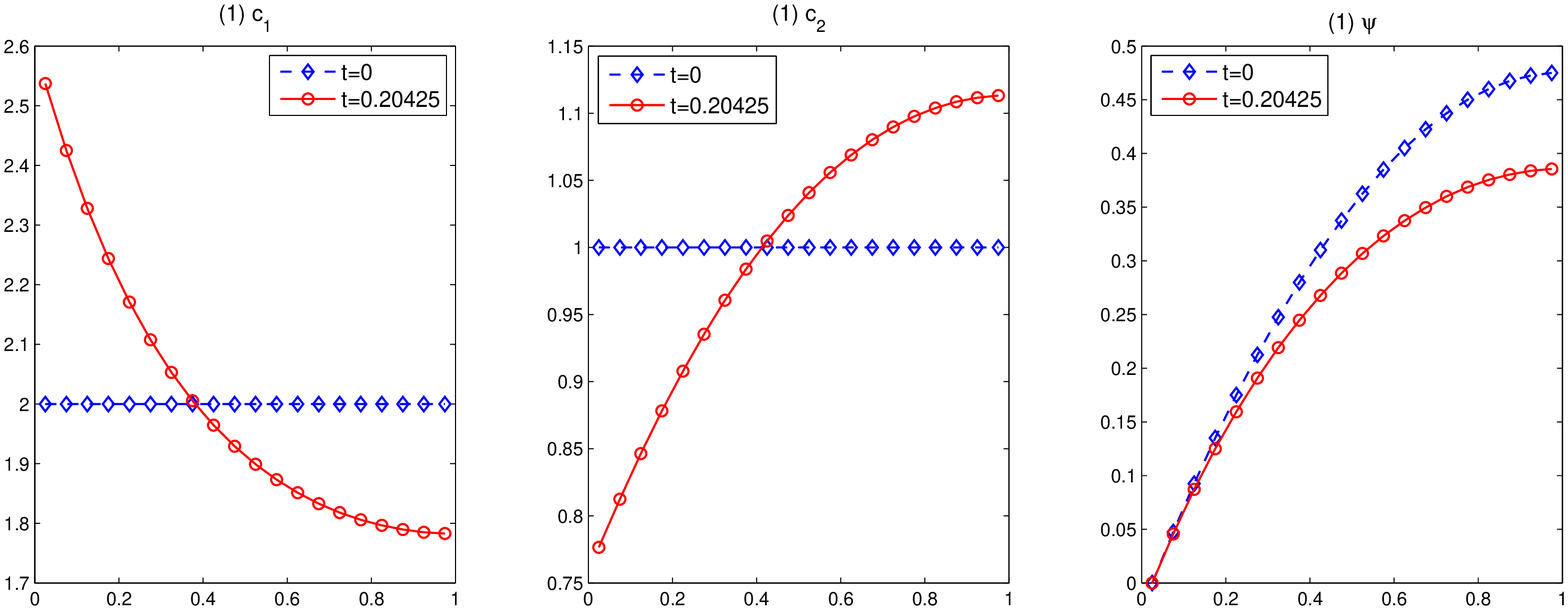}\\
 \includegraphics[width=1.05\textwidth]{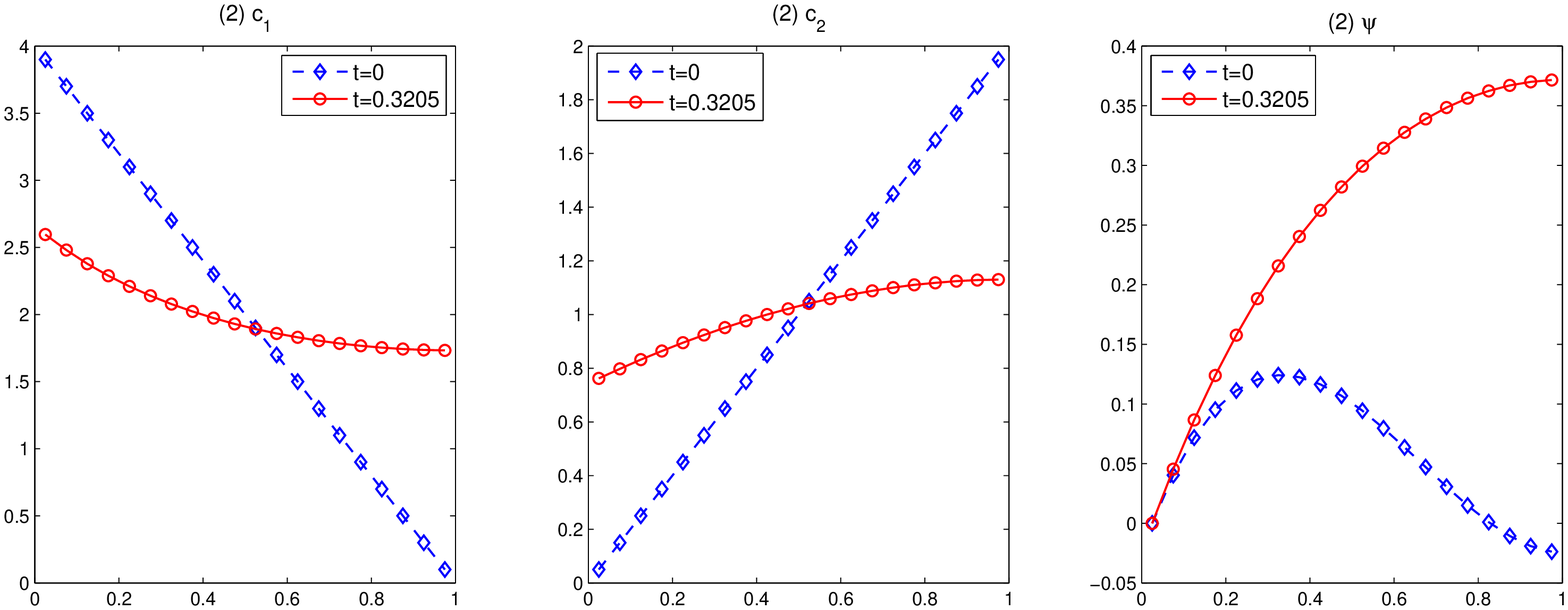}\\
 \includegraphics[width=1.05\textwidth]{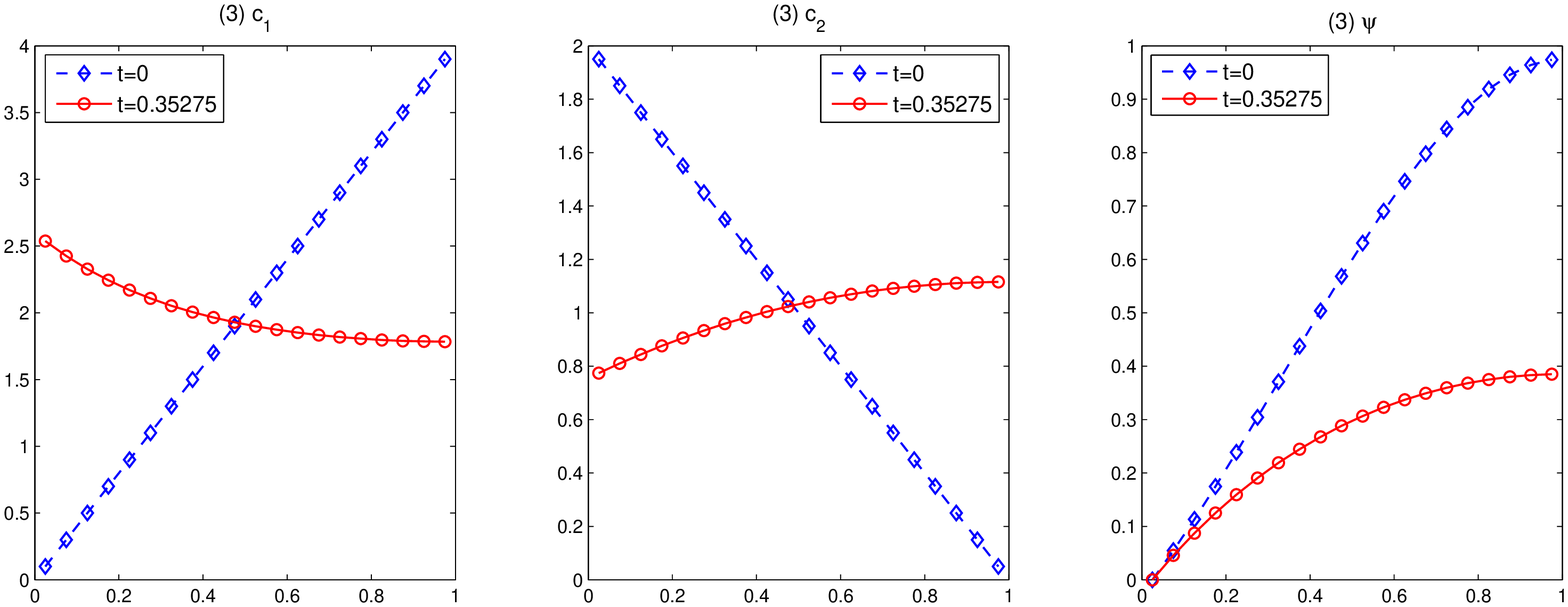}
\end{figure}
\\
\subsection{Two dimensional numerical tests}
\subsubsection{Single species}
We now test our algorithm in a two dimensional setting. Let $\Omega$ be the unit square centered at $(0.5,0.5)$. The initial and boundary conditions are
\begin{align*}
&(1) \quad c(0,x)=4,\quad \left.\frac{\partial \psi}{\partial {\bf n}}\right|_\Gamma=-1; \\
&(2) \quad c(0,x)=2,\quad \left.\frac{\partial \psi}{\partial {\bf n}}\right|_\Gamma= \begin{cases}
             -1 & \text{}  {\{(x,y) |x=1\text{ or }  y=0\}} \\
             0  & \text{else.}
       \end{cases}
\end{align*}
Note that the compatibility condition (\ref{compatibility}) is satisfied for both cases.

Table \ref{2d1errortable} shows the numerical error and the convergence order for the initial case (1), where a reference solution is computed at $h=0.0125$. We observe the convergence rate at order two for $c$ and slightly worse for $\psi$.  Fig. \ref{2d1} shows the steady-state solution of  different initial conditions. For the  case 1, we also display the evolution of the concentration $c$ (from bottom to top as time evolves). The evolution of the potential is not shown since it does not change too much visually.

\begin{table}[htb]
\caption{$l^\infty$ error table of $c$ and $\psi$ for 2D initial case $(1)$ at time $t=0.05$  }   \label{2d1errortable}
\begin{tabular}{|c|c|c|c|c|} \hline
h\footnote{The same h is used for both x and y-directions.}&       error in $c$ & order &error in $\psi$    &order      \\   \hline
          0.2000  &   0.034449 &          - &   0.0005512   &         -   \\   \hline
          0.1000 &    0.009046 &      1.9291&   0.00019733  &      1.482  \\   \hline
         0.0500 &   0.0022608 &      2.0004 & 6.7967e-005   &    1.5377  \\   \hline
        0.0250 &  0.00053914  &     2.0681  & 1.867e-005    &   1.8641  \\   \hline
       0.0125  & 0.00010783  &     2.3219&  5.9918e-006     &  1.6397  \\   \hline
\end{tabular}
\end{table}

\begin{figure}[h]
\caption{ steady-state solutions for different 2D initial cases} \label{2d1}
 \includegraphics[width=0.45\textwidth]{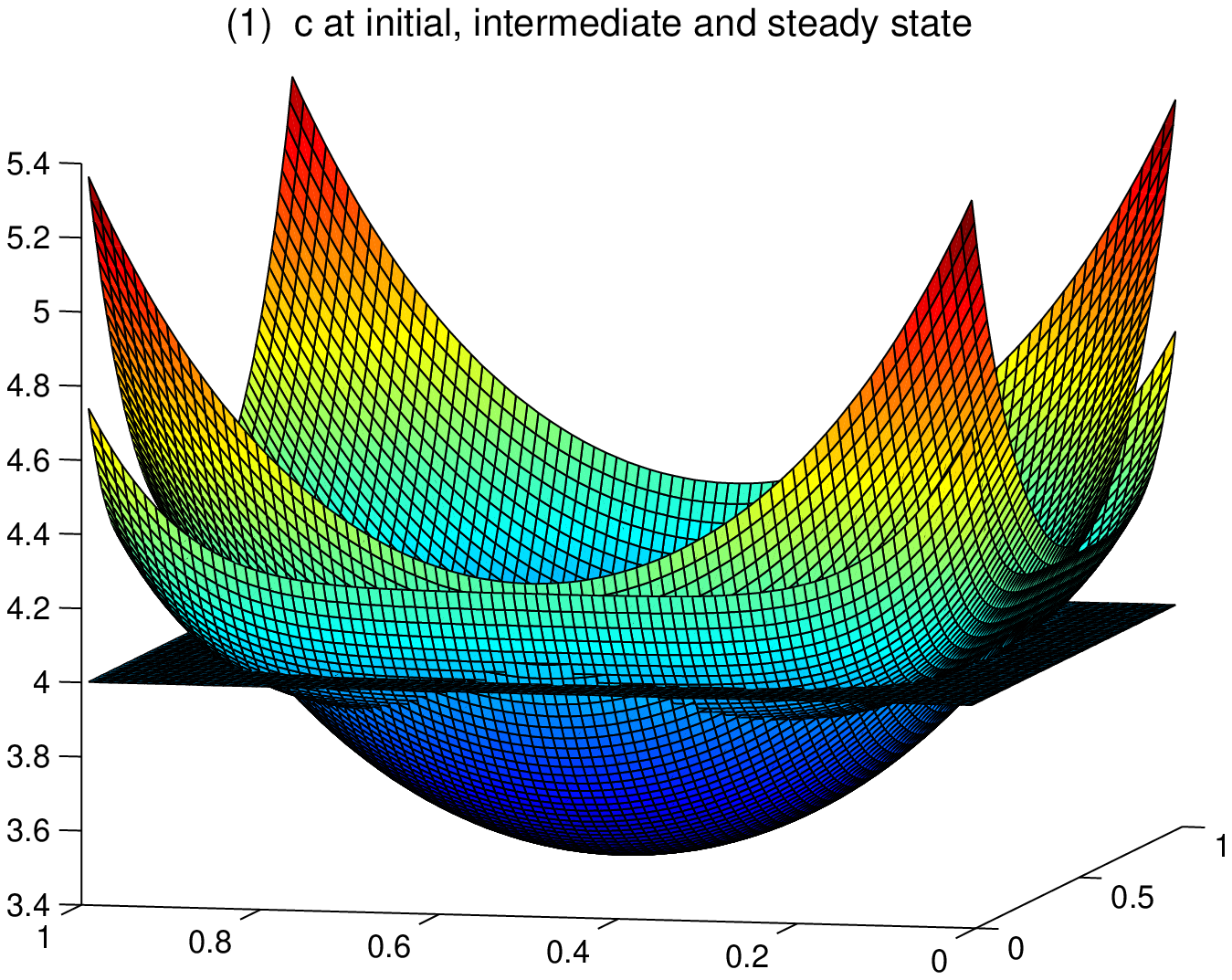}
 \includegraphics[width=0.45\textwidth]{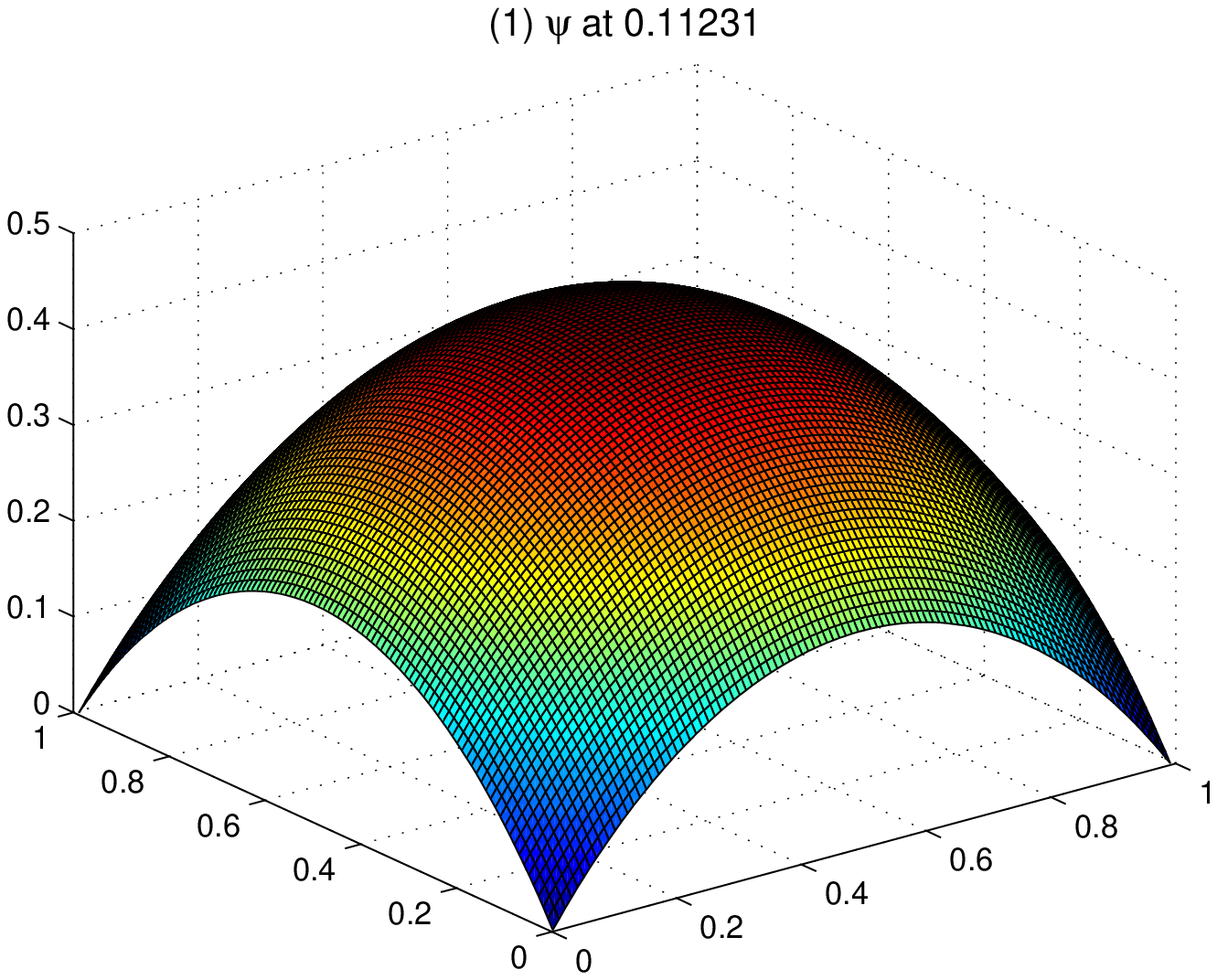}\\
  \includegraphics[width=0.45\textwidth]{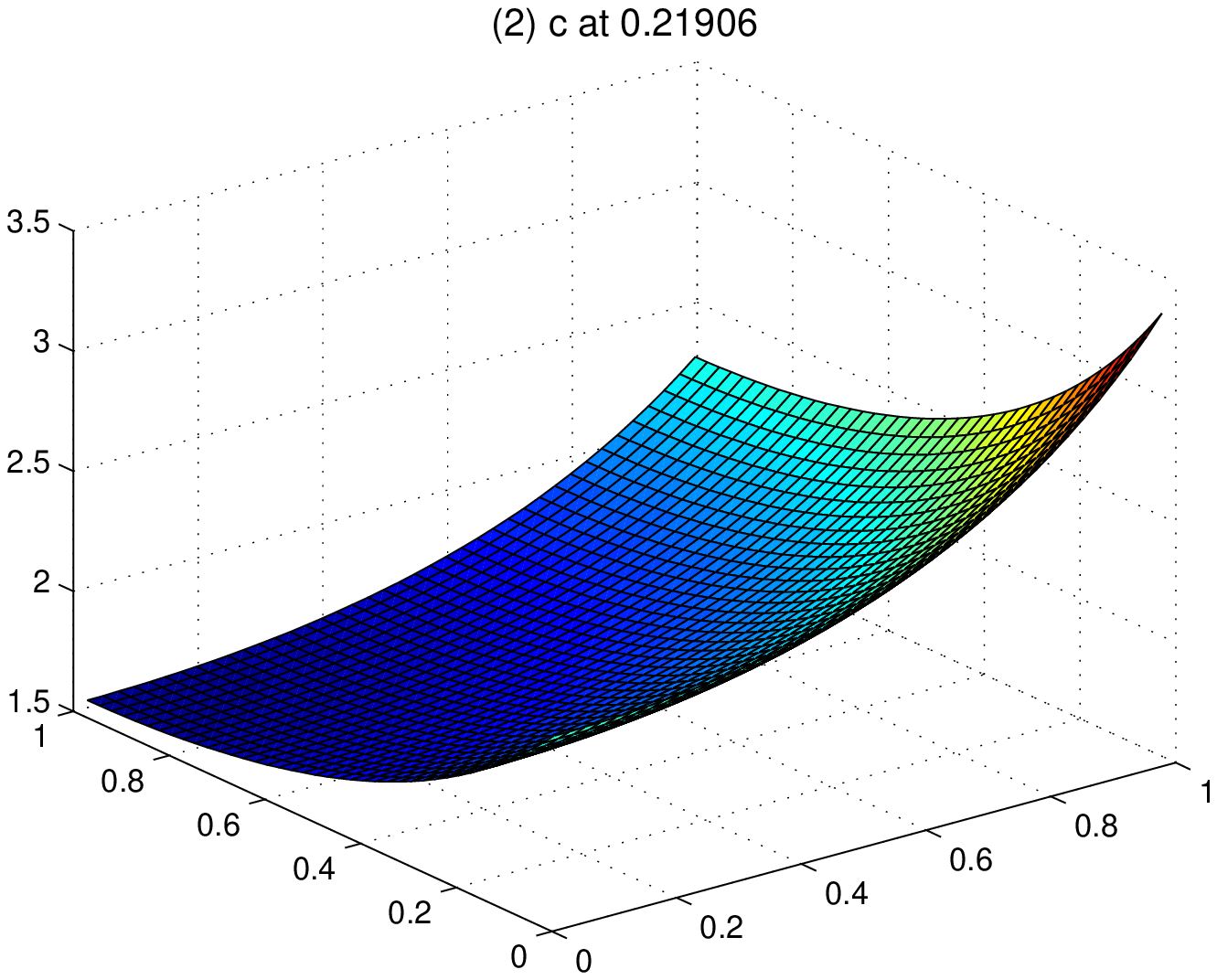} \fontsize{20}{12}\selectfont 
 \includegraphics[width=0.45\textwidth]{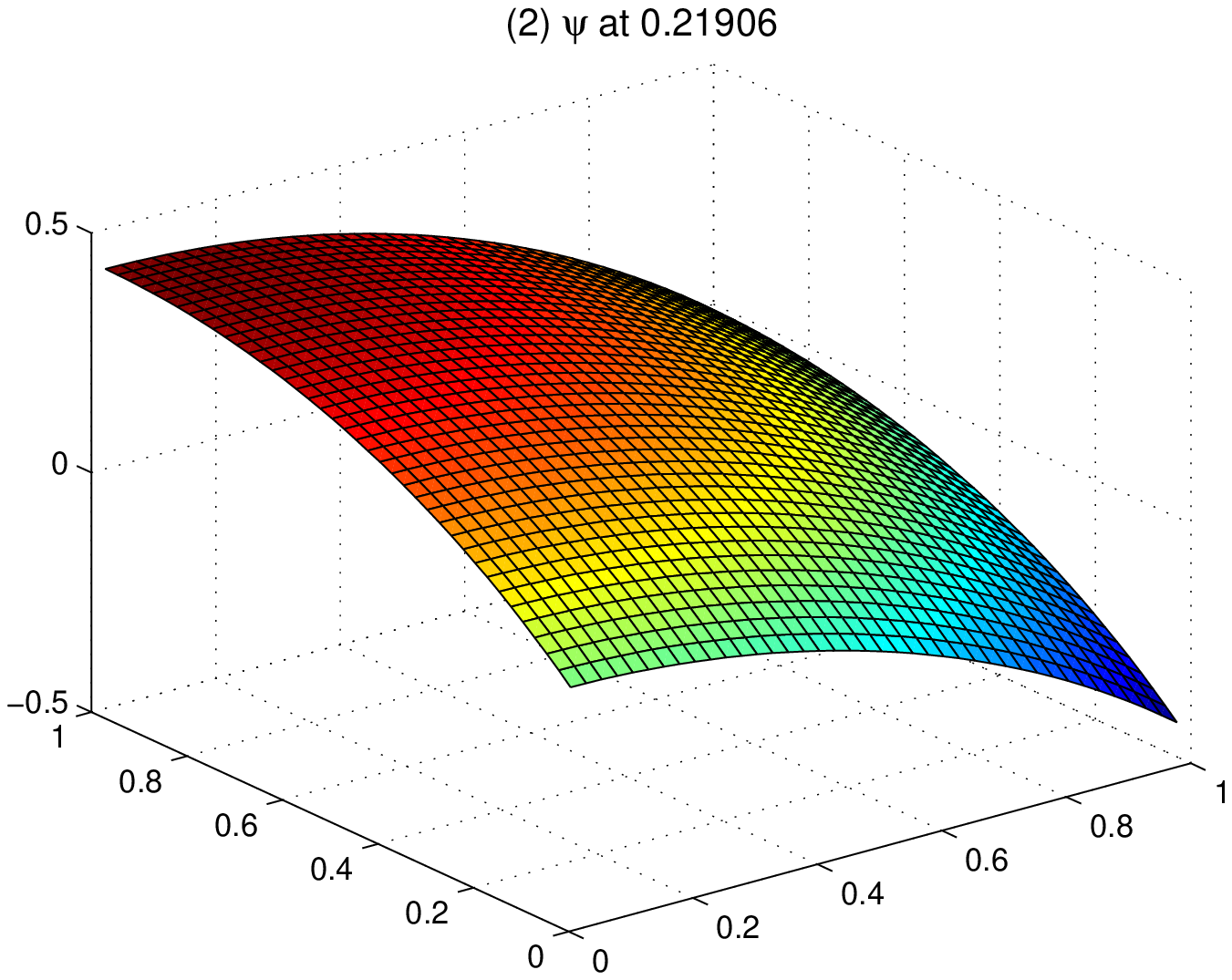}
\end{figure}

\section{Concluding remarks}
In this paper, we have investigated the Poisson-Nernst-Planck equation which is of mean field type model for concentrations of chemical species,  with our focus on the development of a free energy  satisfying numerical method for the PNP equation subject to zero flux for the chemical concentration and non-trivial flux for the potential on the boundary. We constructed simple, easy-to-implement conservative schemes which preserve equilibrium solutions, and proved that they satisfy all three desired properties of the chemical concentration: 
 the total concentrations is conserved exactly,  positivity of the chemical concentrations is preserved under a mild CFL condition $\Delta t/\Delta x <\lambda_0$,  and the free energy dissipation law is satisfied exactly at the semi-discrete level. The analysis is for one and multiple ionic species and for one-dimensional problems. But numerical tests are given for both one-dimensional and two-dimensional systems, as well as for multiple species.

Although this work makes good progress in constructing and analyzing an accurate method for solving the Poisson-Nernst-Planck equations numerically, there remain many challenges, which we wish to address in future work. 
First, we would like to extend the present numerical method and analytical results herein to a higher order discontinuous Galerkin method.
Second, for most ion channels, the appropriate boundary conditions are those with the non-zero flux across the boundary. In such settings, mass conservation is no longer valid,  we will investigate the possibility to preserve the solution positivity and  the energy dissipation property at the discrete level.  


\section*{Acknowledgments} The authors thank Bo Li for stimulating discussions of the PNP system.  Liu's research was partially supported by the National Science Foundation under grant DMS 13-12636. 

\bibliographystyle{plain}

\end{document}